\documentclass[11pt,leqno,a4paper]{amsart}

\usepackage{anysize}
\marginsize{3.5cm}{3.5cm}{2.5cm}{2.5cm}

\usepackage[]{hyperref}
\hypersetup{
    colorlinks=true,       
    linkcolor=red,          
    citecolor=blue,        
    filecolor=magenta,      
    urlcolor=cyan           
}

\usepackage{amsmath}
\usepackage{amsfonts,amssymb}
\usepackage{enumerate}
\usepackage{mathrsfs}
\usepackage{amsthm}
\usepackage{a4wide}

\theoremstyle{plain}
\newtheorem{theo}{Theorem }[section]
\newtheorem{prop}[theo]{Proposition}
\newtheorem{lemm}[theo]{Lemma}
\newtheorem{coro}[theo]{Corollary}

\theoremstyle{definition}

\newtheorem{remas}[theo]{Remarks}

\DeclareSymbolFont{pletters}{OT1}{cmr}{m}{sl}
\DeclareMathSymbol{s}{\mathalpha}{pletters}{`s}

\def\eps{\varepsilon}

\def\mez{\frac{1}{2}}

\def\xC{\mathbf{C}}
\def\xN{\mathbf{N}}
\def\xR{\mathbf{R}}

\numberwithin{equation}{section}

\title{Real analyticity of radiation patterns of resonant states on asymptotically hyperbolic manifolds}
\author{ Claude Zuily (*)}
  \address{Laboratoire de Math\'ematiques UMR 8628 du CNRS. Universit\'e Paris-Sud, B\^atiment 425, 91405 Orsay Cedex.}
   \address{Mail: claude Zuily@math.u-psud.fr}
 \thanks{ (*)  Supported in part by Agence Nationale de la Recherche
   project  ANA\'E ANR-13-BS01-0010-03.}

\begin{document} 
\maketitle
\begin{abstract}
We show that resonant states in scattering on  asymptotically hyperbolic manifolds that are analytic near conformal infinity, have analytic radiation patterns at infinity. On 
even  asymptotically hyperbolic manifolds we also show that smooth solutions of Vasy operators with analytic coefficients are also analytic. That answer a question of M.Zworski (\cite{Z2}  Conjecture 2). The proof is based on previous results of Baouendi-Goulaouic and Bolley-Camus-Hanouzet and for convenience of the reader we present an outline of the proof of the latter.
\end{abstract}
\section{Introduction and statement of the main results.}
  
In this note we consider the question of analyticity of suitably renormalized resonant states on  asymptotically hyperbolic manifolds. Referring to \cite{MM}, \cite{Gu}, \cite{V}, \cite{Z2} (section 3.1), \cite{DZ} (Chapter 5) and \cite{Z1} for detailled presentations and for issues of geometrical invariance, we briefly recall the set-up in the case where the metric is analytic near infinity.
\subsubsection{Radiation patterns}
Let $\overline{M}$ be a compact $n+1$ dimensional manifold with boundary $\partial M \neq \emptyset$ and let $M:= \overline{M} \setminus \partial M.$ We assume that $\overline{M}$ is a real analytic manifold near $\partial M$. The Riemanian manifold $(M,g)$ is said to be {\it asymptotically hyperbolic} and {\it analytic near infinity} if there exist functions $y' \in C^\infty(\overline{M}, \partial M)$ and $y_1 \in C^\infty(\overline{M}, (0,2))$ such that
\begin{align*}
&y_1 \arrowvert_{\partial M} = 0, \quad dy_1 \arrowvert_{\partial M}\neq 0,\\
&\overline{M}\supset  y_1^{-1}([0,1)) \ni m \mapsto (y_1(m),y'(m)) \in [0,1) \times \partial M
\end{align*}
is a real analytic diffeomorphism and near $\partial M$ the metric $g$ has the form
\begin{equation}\label{g}
g\arrowvert_{y_1 \leq \eps} = \frac{dy_1^2 + h(y_1)}{y_1^2}
\end{equation}
where $[0,1) \ni t \mapsto h(t)$ is an analytic family of real analytic Riemanian metrics on $\partial M.$

We recall now the following results of Mazzeo-Melrose \cite{MM} and Guillarmou \cite{Gu}. 

For $\text{Im }\lambda >0$, the operator  $R_g(\lambda) = \big(-\Delta_g - \lambda^2 - \big(\frac{n}{2}\big)^2 \big)^{-1}$  can be  defined from $ L^2(M)$ to $   H^2(M).$ Then 
\begin{equation}\label{mero}
\begin{aligned}
  R_g(\lambda) & :C^\infty_0(M)  \to C^\infty(M)\\ 
 &\text{ continues  to a meromorphic family of operators for }  \lambda \in \xC \setminus (-\frac{i}{2}\xN). 
 \end{aligned}
 \end{equation}

Moreover for $\lambda \in \xC \setminus (-\frac{i}{2}\xN),$
\begin{equation}\label{reson}
\begin{aligned}
u \in \Big(\oint R_g(\zeta)\, &d\zeta\Big) C^\infty_0(M), \quad (-\Delta_g - \lambda^2 - \big(\frac{n}{2}\big)^2 )u=0,\\
& \Longrightarrow F(y):= y_1^{i \lambda - \frac{n}{2}} u \in C^\infty(\overline{M}).
\end{aligned}
\end{equation}
Here the integral is over a small circle enclosing $\lambda$ and no other singularity of $R_g.$

The function $F$ (or $F\arrowvert_{\partial M}$) can be considered as the {\it radiation pattern} of the {\it resonant state} $u$.

In the analytic case it is natural to ask if the radiation patterns are 
{\em real analytic}. This is indeed the case as shown by the following result.   
\begin{theo}\label{t:ress}
Suppose that $ \overline M $ and $ g $ are {\em real analytic} near the
$\partial M $. Then, with the above terminology and for 
$ \lambda \in \xC \setminus (-\frac{i}{2}\xN)$, radiation patterns, $F$, of resonant states $ u $ are real analytic near $ \partial M $. 
\end{theo}

\noindent
\begin{remas}  1. The result is local in the sense that analyticity needs to be 
assumed only at $ m \in \partial M $ with the corresponding local conclusion.

2. An equivalent conclusion would be to say that 
$ F |_{ \partial M} $ is real analytic. That is true in view of Theorems 0.2 and 0.3 which come from [1]. That restriction is what we would normally call the radiation pattern.
\end{remas}
\subsubsection{Vasy operators}
Motivated by analysis of the wave equation for Kerr--de Sitter black hole metrics Vasy [5] introduced a microlocal approach to the meromorphic continuation 
\eqref{mero}. The key component are {\em radial propagation estimates} first obtained by Melrose \cite{Me}.

Vasy's approach works for {\em even asymptotically hyperbolic metrics.} This means that in the notation of \eqref{g} the metric is given by 
\begin{equation}
\label{eq:gash1}   g|_{ y_1 \leq \epsilon }  = \frac{ dy_1^2 + h ( y_1^2 )  }{ y_1^2 } ,  \end{equation}
where 
 $ [ 0, 1 ) \ni t \mapsto h ( t ) $, 
is an analytic family of real analytic Riemannian metrics on $ \partial M $.
In that case 
\begin{equation}
\label{eq:firstconj} y_1^{ i \lambda - \frac n 2 }   ( - \Delta_g -
\lambda^2 - ({\textstyle{\frac n 2}})^2  ) y_1^{-i \lambda + \frac n 2 } =
x_1 P ( \lambda ) , \ \   x_1 = y_1^2 , \ \ x' = y',   , \end{equation}
where, near $ \partial M $,
\begin{equation}
\label{eq:Plag} P ( \lambda ) = 4 ( x_1 D_{x_1}^2 - ( \lambda + i ) D_{x_1} )
- \Delta_h + i \gamma ( x ) \left( 2 x_1 D_{x_1} - \lambda - i
{\textstyle\frac{ n-1} 2 } \right) . \end{equation}
This is the Vasy operator [5]. 

The operator \eqref{eq:Plag} can be considered locally as a special case of
the following class of operators:
\begin{equation}
\label{eq:Vasy}
L:= xD_x^2 -  ( \lambda +i ) D_x + \gamma (x ,y ) x D_x +
Q ( x, y , D_y ), \quad D = \frac{1}{i} \partial,  
\end{equation}
where 
\[  (-1,1) \ni x \mapsto Q ( x, y , D_y ) = \sum_{ |\alpha| \leq 2 } 
a_\alpha ( x, y )
D_y^\alpha ,  \ \ y \in U \subset \xR^n \]
is a an analytic family of  (positive) elliptic second order
differential operators with analytic coefficients.

\begin{theo}\label{hea}
Assume that the coefficients of $L$ are real analytic in a neighborhood of a point $m_0 =(0,y_0) \in (-1,1)\times \xR^n$ and that $\lambda \notin - i\xN^*.$
 
Then if $u$ is a $C^\infty$ function near  $m_0$ such that $Lu$ is real  analytic in a neighborhood of $m_0$ then $u$ itself is real  analytic near $m_0.$
\end{theo}
\begin{remas}  
$(i)$ Notice that the operator $L$ is elliptic for $x>0$ and hyperbolic for $x<0.$

$(ii)$ When $\gamma =0$ and $Q= \Delta_M$ is the Laplace-Beltrami operator on a compact manifold  $M$ the analytic regularity has recently been proved by  Lebeau and  Zworski \cite{LZ}.

$(iii)$ Vasy's adaptation of Melrose's radial estimates shows that
to have $ u \in C^\infty$ we only need to assume that $ u \in H^{s+1} $ near $ m_0 $, where  $ s + {\textstyle{\frac12}} > - \text{Im }  \lambda,$
see \cite{Z1} \S 4, Remark 3. 
\end{remas}
  {\it Acknowledgements.}  Warm thanks are due to Maciej  Zworski for  providing informations on the mathematical background of this note, upon which this introduction was written  and for his great generosity.
\section{Proofs}
\subsection{Preliminaries}
In this section we recall some results by Baouendi-Goulaouic \cite{BG} and   Bolley-Camus-Hanouzet \cite{BCH}.

We begin by a "Cauchy-Kovalevska" type theorem which is a particular case of Theorem 1 in \cite{BG}.

Consider a "Fuchs type"  operator with analytic coefficients near $m_0=(0,y_0)$, of the form
\begin{equation}\label{BG1}
\mathcal{P} =x\partial_x^2 + a(y)\partial_x + \mathcal{Q}(x,y,\partial_y) + \sum_{\vert \beta \vert \leq 1} x \partial_x b_\beta(x,y)\partial_y^\beta 
\end{equation}
where $Q$ is a second order differential operator in $y$ (non necessarily elliptic). To this operator we associate the caracteristic equation,
\begin{equation}\label{cara}
\mathcal{C}(\mu,y) = \mu(\mu-1) + \mu a(y)= \mu(\mu-1+a(y)).
\end{equation}
Then we have the following result.
\begin{theo}\label{CK}
The following properties are equivalent.
\begin{align*}
  (i)  \quad &\text{For any integer } \mu \geq 1 \text{we have } \mathcal{C}(\mu,y_0) \neq 0.\\
 (ii) \quad &\text{For any analytic functions } v_0  \text{ near }   y_0  \text{ and    } f \text{ near } m_0 \text{ there exists a unique function } \\
& v \text{ which is analytic near } m_0  \text{ such that}
\end{align*}
$$\mathcal{P} v = f, \quad v(0, \cdot) = v_0.$$
 \end{theo}
The second result is a "Holmgren type" theorem which is also a particular case of Theorem 2 of \cite{BG}.
\begin{theo}\label{Holm}
Let $\mathcal{P}$ be defined by \eqref{BG1} and $h\in \xN.$ 
Assume that the roots of the equation $\mathcal{C}(\mu, y_0) =0$ satisfy 
$$\text{ Re } \mu_k <1 + h, \quad k=1,2.$$
Then  any function $U$ of class $C^{1+h}$ in $(x,y)$ near $m_0$  satisfying 
$$ \mathcal{P} U = 0, \quad \partial_x^jU(0, \cdot) =0, \quad \text{for } 0\leq j \leq h$$
vanishes identically near $m_0$.
\end{theo}

In the sequel we shall denote by $C^\omega$ the space of analytic functions.
 
Let us recall a particular case of a   result by Bolley-Camus-Hanouzet \cite{BCH}.  \begin{theo}\label{BC} Let $y_0 \in \xR^n$ and    $P_1$ (resp.  $P_2$) be a differential operator  of order $1$ (resp. $2$) with analytic coefficients on $[0,1)\times V_{y_0},$
 \begin{equation}\label{L}
 P_1 = \sum_{\vert \alpha \vert + k \leq 1} a_{\alpha,k}(t,y)D_t^k D_y^\alpha, \quad P_2 = \sum_{\vert \alpha \vert + k =2} b_{\alpha,k}(t,y)D_t^k D_y^\alpha.
 \end{equation}
Assume that $  P_2$   is uniformly  elliptic in  $[0,1)\times V_{y_0}$. Define $L$   by 
$$Lu= P_2(tu) + P_1u.$$
 Then 
 \begin{equation}\label{HEA}
  u\in  C^\infty([0,1)\times V_{y_0}), \,\,     Lu \in  C^\omega([0,1)\times V_{y_0})   \Longrightarrow   u \in C^\omega([0,1)\times V_{y_0}) .
  \end{equation}
    \end{theo}
 For the reader's convenience we shall sketch briefly the proof of this theorem at the end of this note.
 \begin{remas}\label{coroBC}
 
 1.  A previous result of this type is due to Baouendi-Goulaouic \cite{BG2}.

2. Notice that \eqref{HEA} holds without condition on the boundary data.  This fact is due to the degeneracy on the operator $L$ on the boundary $t=0$ and to the particular form of $L.$

3. It follows   that an operator (with analytic coefficients) of the form
 \begin{equation}\label{ex}
 L= t\big( D_t^2 + R_2(t,y,D_y) \big) + R_1(t,y,D_t,D_y)
\end{equation}
where $R_j$ is of order $j$ and $\sigma(R_2)(t,y, \eta) \geq c\vert \eta \vert^2$ satisfies \eqref{HEA}.
\end{remas}
 \subsection{Proof of Theorem \ref{t:ress}}

Using the special form of the metric \eqref{g} we compute (in the coordinates valid near the boundary),
\begin{gather}
\label{eq:Deltagg} 
\begin{gathered} - \Delta_g = ( y_1 D_{y_1} )^2 + 
i ( n  + y_1 \gamma_0 ( y_1, y') ) y_1
D_{y_1} - y_1^2 \Delta_{h (y_1 ) } , \\
\gamma_0 ( t, y') := - {\textstyle{\frac12}} \partial_t \bar h ( t ) / \bar h ( t ) , \ \
\bar h ( t ) := \det h ( t ) , \ \ D := \textstyle{\frac 1 i } \partial .
\end{gathered}
\end{gather} 

To reduce the presentation to a special case of the results by Bolley-Camus-Hanouzet, 
we perform the conjugation which transforms an equation for $ u $ into an equation for $ F$ (in the notation of \eqref{reson}). We can write
\begin{equation}\label{eq:conj1}
\begin{aligned} 
  0 &= y_1^{ i \lambda - \frac n 2 }   ( - \Delta_g -
\lambda^2 - ({\textstyle{\frac n 2}})^2)u = y_1^{ i \lambda - \frac n 2 }   ( - \Delta_g -
\lambda^2 - ({\textstyle{\frac n 2}})^2  ) y_1^{-i \lambda + \frac n 2 }F \\
&  = 
\big\{(y_1 D_{y_1} )^2 + ( i n + i y_1 \gamma_0 - 2 \lambda) y_1 D_{y_1} 
+ y_1 \gamma_0 ( {\textstyle{\frac12}} n - i \lambda ) - y_1^2 \Delta_{h(y_1)}\big\}F \\
& = y_1 \big\{ y_1 ( D_{y_1}^2 - \Delta_{ h( y_1)} ) + f(y) D_{y_1} + g ( y) \big\}F,
\end{aligned}
\end{equation}
where $ f $ and $ g $ are analytic functions. The operators in brackets on the 
right hand side is precisely an operator  of the form \eqref{ex}. We can therefore apply the point 3. in  Remark \ref{coroBC} to conclude.

 \subsection{Proof of Theorem \ref{hea}}
Let  $L$ be the   operator defined by \eqref{eq:Vasy}. Then
$$-L = x \partial_x^2 + (1-i\lambda)\partial_x +i \gamma(x,y) x \partial_x - Q(x,y, D_y).$$
 It is  of the form   \eqref{BG1} with $a_1(y) = 1-i\lambda.$ Therefore $\mathcal{C}( \mu, y_0) = \mu (\mu-i\lambda).$ Since by hypothesis we have $ i\lambda \notin  \xN^*$ we see that the condition    in Theorem \ref{CK} is satisfied.
 
Let $u$ be the $C^\infty$ solution in Theorem \ref{hea}. Assume that we can prove that $u(0, \cdot)$ is analytic near $y_0$. Let $v$ be the analytic solution of the problem
$$ Lv= f, \quad v(0,\cdot) = u(0, \cdot)$$
given by Theorem \ref{CK}. Setting $U = v-u$ we see that $U$ is $C^\infty$ and satisfy
\begin{equation}\label{eq:U}
 \big[x \partial_x^2 + (1-i\lambda)\partial_x +i \gamma(x,y) x \partial_x - Q(x,y, D_y)\big]U = 0, \quad U(0, \cdot) =0.
 \end{equation}
We claim that, when $i\lambda \notin \xN^*,$ we have
\begin{equation}\label{traces}
 \partial_x^jU(0, \cdot) = 0, \text{ for all } j \in \xN.
 \end{equation}
This is true for $j=0$. Assume this is true for $0 \leq j \leq k.$ We differentiate $k$ times the equation in \eqref{eq:U} with respect to $x$ and and we take the trace on $x=0$. Using the induction we see that $(k+1 -i \lambda)\partial_x^{k+1}U(0, \cdot) =0$ which ends the induction.

Now the roots of the equation $\mathcal{C}(\mu, y_0) =0$ are $\mu_1 =0, \mu_2 = i\lambda.$ Taking $h\in \xN$ such that $1+h> \text{Re }\mu_2 = - \text{Im }\lambda$ we see that the condition in Theorem \ref{Holm} is satisfied. Using then \eqref{traces} we see that Theorem \ref{Holm} implies that  $U=0.$ Thus $u = v$ and $u$ is analytic near $m_0$ which completes the proof. Therefore we are left with the proof of the following claim.
\begin{center}
"Under the conditions of Theorem \ref{hea} $u(0,\cdot)$ is analytic near $y_0.$"
\end{center}
Consider now an  operator on $[0, 1)\times \xR^n$ with analytic coefficients     of the form
$$P = xD_x^2 +c(x,y)D_x  +d(x,y) + Q(x,y,D_y)$$
where $\sigma_2(Q)(x,y,\eta) \geq c \vert \eta \vert^2.$

The goal is to prove that  if $u\in C^\infty$ on $[0, 1)\times V_{y_0}$  and $Pu=f$ is analytic on $[0, 1)\times V_{y_0}$ then $u$ is analytic in $[0, 1)\times V_{y_0}$. In particular $u(0,\cdot)$ is analytic.

One can reverse the calculation \eqref{eq:firstconj} and \eqref{eq:conj1}
to see that, with $x=t^2,  $ if $u$ is a solution in $[0,1)$ of the equation $Pu=f$ then $w$ defined by $w(t,y) = u(t^2,y)$ satisfy an equation on $[0,1) \times V_{y_0}$ of the form $\widetilde{P} w =g$ where
$$ g(t,y) = t f(t^2,y) \quad \text{and} \quad \widetilde{P} w  = t P_2 ( t, y, D_t, D_y ) + P_1 (t, y,  D_t , D_y) ,$$ 
where $ P_2 $ is an elliptic operator of order two and $ P_1 $ is a first order
operator -- both with analytic coefficients. Using the point 3. in Remarks \ref{coroBC} we see that $w$ is analytic in $ [0,1)\times V_{y_0}.$ 



To prove that $u(0, \cdot)$ is analytic  we use the following lemma.
 
Let $u:[0, 1) \times V_{y_0} \to \xC$ be  a $C^\infty$ function. Set
\begin{equation}\label{defv}
 w(t,y) = u(t^2,y) \quad0 \leq t  \leq 1, \quad y \in V_{y_0}.
 \end{equation}
\begin{lemm}\label{vtou}
Assume that there exist    positive constants  $A,B_1,B_2$ such that for all  $\alpha \in \xN^n, j \in \xN$ we have
\begin{equation}\label{est:v}
  \vert \partial_y^\alpha \partial_t^j w(t,y)\vert \leq A B_1^{\vert \alpha \vert} B_2^j \alpha! j!, \quad 0 \leq t   \leq 1, \quad   y \in V_{y_0}.
 \end{equation}
Then 
\begin{equation}\label{est:u}
 \vert \partial_y^\alpha \partial_x^j u(x,y)\vert \leq A B_1^{\vert \alpha \vert} B_2^{2j}\alpha! j!, \quad 0 \leq x \leq 1, \quad \vert y \in V_{y_0}.
 \end{equation}
\end{lemm}
\begin{proof}
By definition $w$ is   $C^\infty$ in $[0,1)\times V_{y_0}$  and we have
\begin{equation}\label{der:imp}
\partial_y^\alpha \partial_t^{2\ell +1} w(0,y)=0, \quad \forall \alpha \in \xN^n, \forall \ell \in \xN.
\end{equation}
We claim that we have for all $j \in \xN,$ all $\alpha \in \xN^n$ and all $t\in [0, 1)$
\begin{equation}\label{rec}
  (\partial_x^j \partial_y^\alpha u)(t^2,y)= \frac{1}{2^j} \int_0^1\int_0^1 \cdots \int_0^1 \big(\prod_{\ell = 1}^{j-1}\lambda_\ell^{2(j-\ell)}\big) (\partial_t^{2j}  \partial_y^\alpha  w)(t\prod_{\ell =1}^j \lambda_ \ell ,y) \, \prod_{\ell =1}^{j}d \lambda_\ell.
\end{equation}
where the first product in the integral is equal to $1$ if $j=1$.  

We shall prove this formula for all $\alpha \in \xN^n$ by induction on $j$.
We begin by $j = 1.$  Differentiating both members of \eqref{defv} with respect to $t$ and $y$ we obtain
\begin{equation}\label{j=1}
 2t (\partial_x \partial_y^\alpha  u)(t^2,y) = \partial_t \partial_y^\alpha w(t,y).
 \end{equation}
Using the Taylor formula and \eqref{der:imp} we can write
\begin{equation}\label{j=1+}
 (\partial_t \partial_y^\alpha w)(t,y)= t \int_0^1 (\partial^2_t \partial_y^\alpha w)(\lambda_1t,y)\, d\lambda_1.
 \end{equation}
From \eqref{j=1} and \eqref{j=1+} we obtain for $t\neq 0$
$$(\partial_x \partial_y^\alpha  u)(t^2,y) =Ê\mez \int_0^1 (\partial^2_t \partial_y^\alpha w)(\lambda_1t,y)\, d\lambda_1 $$
and the later formula extends to $t=0$ by continuity. This proves \eqref{rec} for $j=1.$ Assume that \eqref{rec} is true up to the order j and differentiate this equality with respect to $t$. We obtain
\begin{equation}\label{j-j+1} 
 2t(\partial_x^{j+1} \partial_y^\alpha u)(t^2,y) =  \frac{1}{2^j} \int_0^1\int_0^1 \cdots \int_0^1  \prod_{\ell=1}^j\lambda_\ell^{2(j-\ell) +1} 
   \big(\partial_t^{2j+1}  \partial_y^\alpha  w\big)(t \prod_{\ell=1}^j\lambda_\ell,y) \, \prod_{\ell=1}^j d \lambda_\ell.
 \end{equation} 
Using \eqref{der:imp} we can write
$$ \big(\partial_t^{2j+1}  \partial_y^\alpha  w\big)(t \prod_{\ell=1}^j\lambda_\ell,y) = t \prod_{\ell=1}^j\lambda_\ell \int_0^1 \big(\partial_t^{2j+2}  \partial_y^\alpha  w\big)(t \prod_{\ell=1}^{j+1}\lambda_\ell,y)\, d\lambda_{j+1}.$$
Plugging the right hand side into the integral \eqref{j-j+1} and dividing both members by $t$ for $t \neq 0$ we obtain \eqref{rec} for $j+1$ and $t\neq 0$ and also for $t=0$ by continuity.

We can now prove the lemma. Indeed using \eqref{rec} and \eqref{est:v} we obtain for $x\in [0,1)$
\begin{align*}
 \vert (\partial_x^j \partial_y^\alpha u)(x,y)\vert &\leq A B_1^{\vert \alpha \vert} B_2^{2j} \alpha!(2j)! \frac{1}{2^j}\prod_{k =1}^{j-1} \frac{1}{2k+1}\\
 & \leq  A B_1^{\vert \alpha \vert} B_2^{2j} \alpha!(2j)! \frac{1}{2^j} \frac{2^j j!}{(2j)!}
 \end{align*}
 which proves \eqref{est:u}.The proof is complete.

 An even simpler proof of the analyticity of the trace of $u$ would be to consider the Taylor series of $w$ and to use \eqref{der:imp}.
\end{proof}
\subsection{Sketch of the proof of Theorem \ref{BC}.}
We follow closely \cite{BCH}. We shall be using the following version of Hardy's inequality.
\begin{lemm}\label{Hardy}
For $u\in C^\infty,$ with compact support in $[0, + \infty)\times \xR^n$ $ k\in \xN, \beta \in \xN^n$ we have
\begin{equation}\label{H:ineq}
\Vert D_t^k D_y^\beta u \Vert_{L^2(\xR_+ \times \xR^n)} \leq \frac{2}{2k +1} \Vert D_t^{k +1}D_y^\beta (tu) \Vert_{L^2(\xR_+ \times \xR^n)}.
\end{equation}
\end{lemm}
\begin{proof}
Set  $v = D_y^\beta u, w = \partial_t (tv).$  We have
 $v(t,y) = \frac{1}{t} \int_0^t w(s,y)\, ds = \int_0^1 w(tx,y)\, dx.$ 
Differentiating this equality $k$ times with respect to $t$ we get
 $\partial_t^k v(t,y) =  \int_0^1 x^k (\partial_s^k w)(tx,y)\, dx$ 
from which we deduce
    $$\Vert \partial_t^k v\Vert_{L^2(\xR_+ \times \xR^n)}  \leq  \int_0^1 x^k \Vert (\partial_s^k w)( x \cdot ,\cdot) \Vert_{L^2(\xR_+ \times \xR^n)}\, dx 
    \leq \Big(\int_0^1 x^{k-\mez}\, dx \Big)  \Vert \partial_t^k w\Vert_{L^2(\xR_+ \times \xR^n)} 
$$
  which completes the proof.
\end{proof}

The first step of the proof is the following a priori estimate.  
\begin{prop}
Let $y_0 \in \xR^n.$ There exists $p_0 \in \xN$ such that for all $p\in \xN$ with $p\geq p_0$ there exist $C_p>0, \eps_p>0$ such that 
\begin{equation}\label{est:apr} 
 \Vert tu \Vert_{H^{p+2}(\xR_+ \times \xR^n)} \leq C_p \big\{ \Vert Lu \Vert_{H^{p}(\xR_+ \times \xR^n)} + \Vert tu \Vert_{H^{p+1}(\xR_+ \times \xR^n)}\big \}
 \end{equation}
for every $u$ such that    $\text{supp }u \subset \{(t,y):  0\leq t < \eps_p, \vert y-y_0 \vert < \eps_p\} $ and the right hand side is finite.
\end{prop}
\begin{proof}
This inequality is the estimate  $(1.7)$ in \cite{BCH} since, with their notations we have $\chi_p =0.$ To explain this last point we follow \cite{BC} (see conditions $H_1(p;\Omega), H_2(p;\Omega)$). We have $\chi_p = 1- r_p$ where $r_p$ is the number of roots of the characteristic equation $ ib_{0,2}(0,y)\rho + a_{0,1}(0,y)= 0$ (with the notations in \eqref{L}) such that $\text{Re } \rho (y) > -p - \frac{3}{2}.$ Taking $p$ large enough we find that $r_p = 1$ so that $\chi_p =0.$ Therefore no boundary condition is required.

We only review the main points of the proof of \eqref{est:apr} referring   to \cite{BC}  for the details.
In the first step (the main one) one consider a one dimensional constant coefficient operator of the form
$$Lu = P_2(D_t)(tu) + P_1(D_t)u$$
where $P_2$ is of order $2$ and $P_1$ of order $\leq 1.$ Introduce for $m\geq 1$ the space
$$\mathcal{H}^m= \{u \in H^{m-1}(\xR_+), tu \in H^m(\xR_+)\}$$
where $H^k(\xR_+)$ is the usual Sobolev space. Assuming that
$$P_2(\tau) \neq 0, \quad \forall \tau \in \xR,$$
it is proved that there exists $p_0 \in \xN$ such that for all $p\geq 0$ the operator $L$ is an isomorphism from $\mathcal{H}^{p+2}$ onto a subspace of $H^p(\xR_+)$ of codimension one. To prove  this fact they  apply the "Fuchs theory" (see \cite {CL}) to the operator $\widehat{L} = P_2(\tau)(-D_\tau) + P_1(\tau)$ obtained (in spirit) from $L$ in taking  the Fourier transform in $t$. This leads to an inequality in some appropriate spaces for $\widehat{L} $ and then for $L.$

In a second step they consider a partial differential operator with constant coefficients of the form
$$Lu = P_2(D_t,D_y) (tu) + P_1(D_t, D_y)u,$$
where $P_2$ is assumed to be elliptic. Performing a Fourier transform with respect to $y$ we reduce ourselves to a one dimensional operator to which we apply the first step.

Finally one consider the variable coefficient case where
$$Lu =  P_2(t,y,D_t,D_y) (tu) + P_1(t,y, D_t, D_y)u,$$
where $P_2$ is assumed to be elliptic. We write $Lu=L_0u +L_1u$ with
 $$L_0u =  P_2(0,y_0,D_t,D_y) (tu) + P_1(0,y_0, D_t, D_y)u,$$
 we use the second step to deal with $L_0,$   the smallness of the coefficients of $L_1$ and the Hardy inequality \eqref{H:ineq} to end the proof of  \eqref{est:apr}.
\end{proof}

Next we use the classical method of nested open sets introduced by Morrey-Nirenberg in \cite{MN}. Let $a \in ]0,1]$ and $\omega$  be  a neighborhood of $y_0$  such that 
$$\Omega  =: \,  [0,a[ \times \omega \subset \{(t,y): 0\leq t <\eps_p, \vert  y - y_0  \vert< \eps_p\}.$$
For $0<\eps<a$ we set 
$$\omega_\eps = \{y \in \omega : d(y, \omega^c) >\eps\}, \quad \Omega_\eps =  [0,a-\eps [ \times \omega_\eps, \quad N_\eps(u) = \Vert u \Vert_{L^2(\Omega_\eps)}.$$
For  $\eps>0$ and $\eps_1>0 $ small enough there exists a function $\psi \in C_0^\infty(\Omega_{\eps_1})$ such that
\begin{equation}\label{est:psi}
\begin{aligned}
 &(i) \quad  \psi(t,y) = 1 \text{ if } (t,y) \in \Omega_{\eps + \eps_1},\\  
 & (ii) \quad \Vert D_t^j D_y^\alpha \psi\Vert_{L^2(\Omega_{\eps_1})} \leq C \eps^{-j-\vert \alpha \vert}, \quad \forall j\in \xN, \forall \alpha \in \xN^n,
\end{aligned}
\end{equation}
with a constant $C$ independent of $\eps, \eps_1.$

\begin{lemm}\label{step1}
There exists  $C>0$ such that for   all $\eps >0, \eps_1>0$ such that $\eps + \eps_1 <a, $   all $j \in \xN, \nu \in \xN^n$ with $j + \vert \nu \vert \leq p+2$ and all  $u\in C^\infty(\xR \times \xR^n),$ we have
$$N_{\eps + \eps_1}(D_t^j D_y^\nu (tu)) \leq C \Big\{ \sum_{\ell + \vert \beta \vert \leq p }\eps^{-p +\ell + \vert \beta \vert }ÊN_{\eps_1}( D_t^\ell D_y^\beta Lu) +  \sum_{\ell + \vert \beta \vert \leq p +1}Ê\eps^{-p -2 +\ell + \vert \beta \vert }N_{\eps_1}( D_t^\ell D_y^\beta (tu)) \Big\}$$
\end{lemm}
\begin{proof}
We apply \eqref{est:apr} to $\psi u.$    We write $L(\psi u) = \psi Lu + [L, \psi]u.$  There is only one term in the commutator which does not contain $tu$.  With the notations in \eqref{L} it can be written  $[P_1, \psi]u $ and has to be estimated in the $H^p$ norm. Essentially one has to estimate   terms of the form $A=: \Vert (D^a  D\psi)  D^b u  \Vert_{L^2}$ (where $D$ is a derivation and $\vert a \vert + \vert b \vert \leq p$) by the right hand side of the inequality in the lemma.
According to \eqref{est:psi} and \eqref{H:ineq} we can write 
$$A \leq C \eps^{-\vert a \vert -1}\Vert D^b u \Vert_{L^2(\Omega_{\eps_1})}\leq C' \eps^{-\vert a \vert -1}\Vert D_t D^b (t u) \Vert_{L^2(\Omega_{\eps_1})} \leq C'' \sum_{\vert c  \vert \leq p+1} \eps^{ -p-2 + \vert c \vert} N_{\eps_1}(D^c(tu)).$$
  \end{proof}
We begin now to estimate the derivatives of higher order.
\begin{lemm}\label{step2}
There exist   $C>0, K>0$ such that for   all $\eps >0, \eps_1>0$ such that $\eps + \eps_1 <a, $   all $j \in \xN, \nu \in \xN^n$ with $j + \vert \nu \vert \leq p+2$ all $\alpha \in \xN^n$ and all  $u\in C^\infty(\xR \times \xR^n),$ we have
\begin{center}
  $N_{\eps + \eps_1}(D_t^j D_y^{\nu+ \alpha} (tu)) \leq C(A_1 +A_2 +A_3),$
  \end{center}
      \begin{align*}
    & A_1 =  \sum_{\ell + \vert \beta \vert \leq p }\eps^{-p +\ell + \vert \beta \vert }ÊN_{\eps_1}( D_t^\ell D_y^{\beta + \alpha}Lu),  \qquad 
   A_2  \sum_{\ell + \vert \beta \vert \leq p +1}Ê\eps^{-p -2 +\ell + \vert \beta \vert }N_{\eps_1}( D_t^\ell D_y^{\beta + \alpha} (tu)),\\
  &A_3=  \sum_{\vert \beta \vert + \ell \leq p}  \sum_{\vert \delta \vert + r\leq 2} \eps^{-p +\ell + \vert \beta \vert }\sum_{\substack{\alpha_1 \leq \alpha\\ \alpha_1 \neq 0} }\sum_{\beta_1 \leq \beta}  \sum_{\ell_1 \leq \ell}  \binom{\alpha}{\alpha_1}     \binom{\beta}{\beta_1}\binom{\ell}{\ell_1}    K^{\ell_1 + \vert \alpha_1 \vert + \vert \beta_1 \vert}(\ell_1 + \vert \alpha_1 \vert + \vert \beta_1 \vert)! \\
  &\eps_1^{-(\ell_1 + \vert \alpha_1 \vert + \vert \beta_1 \vert)}
     N_{\eps_1}( D_t^{r + \ell - \ell_1} D_y^{\delta + \beta - \beta_1 + \alpha- \alpha_1}(tu)).                                  
 \end{align*}
 
\end{lemm}
\begin{proof}
We apply Lemma \ref{step1} to $ D_y^\alpha u.$ and we have to prove that   $ N_{\eps_1}( D_t^\ell D_y^{\beta}[L, D_y^\alpha]u) $ can be bounded by the term $A_3$ appearing in the statement. According to \eqref{L} we have $[L, D_y^\alpha]u = [P_2, D_y^\alpha] tu  + [P_1, D_y^\alpha]u.$ Now $P_1$ is a finite sum of terms of the form $aD_t + bD_{y_j}.$  We have $D_t^\ell D_y^{\beta}[ D_y^\alpha, aD_t]u = D_t^\ell D_y^{\beta}[ D_y^\alpha, a]D_tu.$ Then the desired estimate follows from  the Leibniz formula, \eqref{est:psi},  the estimate
$$\sup_{\omega_{\eps_1}}\vert D^\alpha a \vert \leq K^{\vert \alpha \vert}\alpha! \eps_1^{- \vert \alpha \vert}$$
(which follows from the analyticity of the coefficient $a$) and  \eqref{H:ineq}. All the other terms are estimated by the same way.

Now we   use the fact that our $C^\infty$ solution $u$ is such that $Lu$ is analytic near the point $(0,y_0).$
\begin{coro}
For any integer $p \geq p_0$ one can find $M>0$ such that for every $\alpha \in \xN^n,$   every integer $\ell \leq p+2$ and every $\eps \in ]0, \frac{a}{ \ell + \vert \alpha \vert}[$ we have
\begin{equation}\label{step2}
N_{(\ell + \vert \alpha \vert) \eps} (D_t^\ell D_y^\alpha (tu)) \leq M^{\ell + \vert \alpha \vert +1} \eps^{-(\ell + \vert \alpha \vert )}.
\end{equation}
 \end{coro}
\begin{proof}
We use an induction on   $j = \ell + \vert \alpha \vert.$ Since $u$ is $C^\infty$ and $\eps<1$ we may assume that \eqref{step2} is true for   $j \leq p+1$.  Assume it is true up to the order $j$ and let $(\ell', \alpha') \in \xN \times \xN^n$ be such that $\ell' \leq p+2, \ell' + \vert \alpha' \vert = j+1.$ We   write $ \alpha' = \alpha + \nu$ with $\vert \nu \vert = p+2-\ell',$ then $\vert \alpha \vert = j - p -1.$ Applying  Lemma \ref{step2} with $0< \eps < \frac{a}{j+1}, \eps_1 = j\eps $ we obtain
 $N_{(j+1)\eps}(D_t^{\ell'} D_y^{\alpha + \nu}(tu)) \leq C \sum_{i=1}^3 A_i.$ 
 
 We have $ A_1 =  \sum_{\ell + \vert \beta \vert \leq p }\eps^{-p +\ell + \vert \beta \vert }ÊN_{ j \eps }( D_t^\ell D_y^{\beta + \alpha}Lu).$ Using the analyticity of $Lu$ and the fact that $\vert \alpha \vert = j - p -1$ and $j\eps <a$ we obtain  
 $$A_1 \leq \eps^{-(j+1)} \sum_{\ell + \beta \leq p}K_1^{ j+ \ell +\vert \beta\vert -p}(j+ \ell +\vert \beta\vert -p-1)! \, \eps^{j+ \ell +  \vert \beta \vert  -p +1}.$$
 Since $k! \leq k^k$ we have $(j+ \ell +\vert \beta\vert -p-1)! \, \eps^{j+ \ell + \vert \beta \vert   -p +1} \leq (j\eps)^{j+ \ell +\vert \beta \vert   -p +1}\leq a^{j + \ell +\vert \beta \vert   -p +1}. $ It follows that there exists a constant $M_2$  such that 
 \begin{equation}\label{A1}
  A_1 \leq K_2^{j+2}\eps^{-(j+1)}.
  \end{equation}
  To estimate $A_2$ we use the induction. We have $\eps_1 = j \eps.$ Moreover in the sum we have $\ell + \vert \beta \vert \leq p+1.$ Therefore $\ell + \vert \beta \vert + \vert \alpha \vert \leq p+1 + j -p-1= j.$ By the induction 
  $A_2 \leq \sum_{\ell + \vert \beta \vert \leq p+1} \eps ^{-p-2+\ell + \vert \beta \vert} M^{\ell + \vert \beta \vert + \vert \alpha \vert +1} \eps^{-(\ell +\vert \beta \vert + \vert \alpha \vert)}.  $ Since $ \vert \alpha \vert + p+ 2= j+1$ we can write
  \begin{equation}\label{A2}
   A_2 \leq M^{j+1} \eps^{-(j+1)}  \Big(\sum_{\ell + \vert \beta \vert \leq p+1} 1\Big)\leq \frac{C}{M} M^{j+2} \eps^{-(j+1)},
   \end{equation}
 where $C$ depends only on $p$ and $n$. 
 
 Consider $A_3$. Since $r + \vert \delta \vert + \ell+ \vert \beta \vert - \vert \beta_1 \vert + \vert \alpha \vert - \vert \alpha_1 \vert \leq 2+p+ \vert \alpha \vert - \vert \alpha_1 \vert \leq p+1 + \vert \alpha \vert \leq j$ (since $\alpha_1 \neq 0$) we can use the induction. Recall $\vert \alpha \vert  = j-p-1 $ and $\eps_1 = j \eps.$
 We obtain a term of the form $\eps^N$ with 
 \begin{align*}
 N&= -p + \ell + \vert \beta \vert - (r + \vert \delta \vert + \ell - \ell_1+ \vert \beta \vert - \vert \beta_1 \vert+ \vert \alpha \vert  - \vert \alpha_1  \vert)- (\ell_1 +\vert \alpha_1 \vert + \vert \beta_1 \vert ) = \\
  &= -p -(r + \vert \delta) - \vert \alpha \vert = -(j +1) + 2 -(r + \vert \delta)\geq -(j +1). 
  \end{align*}
  Therefore $\eps^N \leq \eps^{-(j+1)}.$
 
  We will have a constant of the form $ K_3^{\ell_1 + \vert \alpha_1\vert + \vert \beta_1 \vert +1}M^C$ with 
  \begin{align*}
   C&= r + \vert \delta \vert + \ell + \vert \beta \vert + j-p-1-\ell_1 - \vert \alpha_1 \vert - \vert \beta_1 \vert +1,\\
   &=(r + \vert \delta \vert -2) + (\ell + \vert \beta \vert -p)+ j+2 -(\ell_1 +\vert \alpha_1 \vert + \vert \beta_1 \vert) \leq  j+2 -(\ell_1 +\vert \alpha_1 \vert + \vert \beta_1 \vert).
\end{align*}  
  This constant is therefore bounded by  $M^{j+2} K_3\big(\frac{K_3}{M}\big)^{\ell_1 +\vert \alpha_1 \vert + \vert \beta_1 \vert}\leq  M^{j+2}K_3\big(\frac{K_3}{M}\big)^{ \vert\alpha_1 \vert}$ if $M \geq K_3 .$
 Since $  \ell +  \vert \beta \vert \leq p+2$ the sums in $\ell_1 $ and $\beta_1$ are bounded by a fixed constant depending only on $p$ and the dimension $n.$ Therefore we are left with the sum in $\alpha_1$. Now since $j = \vert \alpha \vert + \ell -1$ we have
 $$ \binom{\alpha}{ \alpha_1} \vert \alpha_1 \vert ! j^{-\vert  \alpha_1 \vert} \leq 1.$$
 Summing up we have proved that
 \begin{equation}\label{A3}
   A_3 \leq C_{p,n}M^{j+2} \eps^{-(j+1)} K_3\sum_{\substack{ \alpha_1 \leq \alpha\\ \alpha_1 \neq 0}}\big(\frac{K _3}{M}\big)^{ \vert\alpha_1 \vert} \leq   C'_{p,n} \frac{K^2_3}{M}M^{j+2} \eps^{-(j+1)} 
   \end{equation}
  if $M \geq 2K_3.$  Taking $M$ still larger, compared  to the fixed constants $K_1,K_2,K_3,C_{p,n},C'_{p,n},$ we  deduce  from \eqref{A1},\eqref{A2},\eqref{A3} that \eqref{step2} is satisfied with the same constant $M$ when $\ell + \vert \alpha \vert = j+1.$
    \end{proof}
\begin{coro}
For any integer $p \geq p_0$ one can find $H>0$ such that for every $\alpha \in \xN^n,$   every integer $\ell \leq p+2$ we have
\begin{equation}\label{step3}
N_{\frac{a}{2}} (D_t^\ell D_y^\alpha(tu)) \leq H^{\ell + \vert \alpha \vert +1} (\ell + \vert \alpha \vert)!
\end{equation}
 \end{coro}
\begin{proof}
In \eqref{step2}   take $\eps = \frac{a}{2 (\ell + \vert \alpha \vert)}.$ We obtain $N_{\frac{a}{2}} (D_t^\ell D_y^\alpha(tu)) \leq (2M)^{ \ell + \vert \alpha \vert   +1} a^{-(\ell + \vert \alpha \vert)}( \ell + \vert \alpha \vert )^{ \ell + \vert \alpha \vert}$ and we conclude using the Stirling formula.
\end{proof}
The purpose now is to prove an estimate of type \eqref{step3} for all derivatives of $tu.$ The operator $P_2$ being elliptic we can assume that the coefficient of $D_t^2$ is equal to one. Therefore one can write 
\begin{equation}\label{rec1}
 D_t^2(tu) = Lu - \sum_{\substack{\vert \alpha \vert + j \leq 2\\j \neq 2}}a_{j\alpha}(t,y)D_t^j D_y^\alpha (tu) - \sum_{ \vert \alpha \vert + j \leq 1 }b_{j\alpha}(t,y)D_t^j D_y^\alpha  u.
 \end{equation}

\begin{prop}
One can find $M_1>0, M_2>0$ such that for all $(\alpha, \ell)  \in \xN^n \times \xN$ we have
\begin{equation}\label{step4}
N_{\frac{a}{2}} (D_t^\ell D_y^\alpha(tu)) \leq M_1^{ \vert \alpha \vert +1} M_2^\ell  (\ell + \vert \alpha \vert)!
\end{equation}
 \end{prop}
\begin{proof}
The proof is quite technical. Here we only give the main ideas; details can be found in \cite{BCH}. 
By \eqref{step3} the estimate \eqref{step4} is true for $\ell \leq p+2$ and  every $\alpha \in \xN^n.$  Assume it is true for every $\alpha \in \xN^n$ and for $\ell \geq p+2$ and let us prove it for $(\alpha, \ell  +1) $ were $p$ is very large. We apply the operator $ D_t^{\ell-1} D_y^\alpha$ to both members of \eqref{rec1}. Very roughly speaking $N_{\frac{a}{2}} (D_t^{\ell+1} D_y^\alpha(tu))$ will be bounded by terms (besides   derivatives of the coefficients which are under control)  of the following type.
\begin{align*}
&(0)= N_{\frac{a}{2}} (D_t^{\ell-1} D_y^\alpha(Lu)), \qquad
   (1) = N_{\frac{a}{2}}( D_t^{\ell-1}D_y^\beta (tu)),\quad \vert \beta \vert = \vert \alpha \vert +2,\\
 &(2) = N_{\frac{a}{2}}(D_t^{\ell}D_y^\gamma (tu)),\quad \vert \gamma \vert = \vert \alpha \vert +1,  \qquad(3) = N_{\frac{a}{2}}(D_t^{\ell-1}D_y^\gamma ( u)),\quad \vert \gamma \vert = \vert \alpha \vert +1,\\
 &(4) = N_{\frac{a}{2}}(D_t^{\ell}D_y^\alpha ( u)) 
    \end{align*}
and we want an estimate of the left hand side by $M_1^{\vert \alpha \vert +1} M_2^{\ell +1} (\vert \alpha \vert + \ell +1)!$  

Since $Lu$ is analytic one can find $A_1,A_2$ such that  $(1) \leq A_1^{\vert \alpha \vert +1 } A_2^{\ell -1} (\vert \alpha \vert + \ell -1)! .$ Now the induction shows that
\begin{align*}
 &(2) \leq M_1^{\vert \alpha \vert +3} M_2^{\ell -1} (\vert \alpha \vert + \ell +1)! = \frac{M_1^2}{M_2^2} M_1^{\vert \alpha \vert +1} M_2^{\ell +1} (\vert \alpha \vert + \ell +1)!,\\
 &  (3)   \leq M_1^{\vert \alpha \vert +2} M_2^{\ell} (\vert \alpha \vert + \ell +1)! = \frac{M_1 }{M_2 } M_1^{\vert \alpha \vert +1} M_2^{\ell +1} (\vert \alpha \vert + \ell +1)! .
 \end{align*}
 Now by Lemma \ref{Hardy} and the induction one can write
 $$(3) \leq \frac{2}{2 \ell -1} N_{\frac{a}{2}}(D_t^{\ell }D_y^\gamma ( tu))    \leq M_1^{\vert \alpha \vert +2} M_2^{\ell} (\vert \alpha \vert + \ell +1)! = \frac{M_1 }{M_2 } M_1^{\vert \alpha \vert +1} M_2^{\ell +1} (\vert \alpha \vert + \ell +1)! .$$
 Using again Lemma \ref{Hardy} we obtain   
 $$(4) \leq \frac{2}{2 \ell +1} N_{\frac{a}{2}}(D_t^{\ell +1}D_y^\alpha( tu)) $$
 and taking $\ell$ so large that $ \frac{2C}{2 \ell +1}\leq \mez$ we can absorb this term by the left hand side.
 
 Chosing $A_1 \leq \eps M_1, A_2 \leq   M_2 $ and $\frac{M_1 }{M_2 } \leq \eps, $ $\eps $ small, we obtain eventually
 $$N_{\frac{a}{2}} (D_t^{\ell+1} D_y^\alpha(tu)) \leq C_0 \eps M_1^{\vert \alpha \vert +1} M_2^{\ell +1} (\vert \alpha \vert + \ell +1)!  $$
 where $C_0$ is an absolute constant.
It suffices to take $\eps$ such that $C_0 \eps \leq 1$ to conclude.
\end{proof}
It follows from \eqref{step4} and Lemma \ref{Hardy} that $u$ is analytic in a neighborhod of the point $(0,y_0)$,  which completes the sketch of proof of Theorem \ref{BC}.

\end{proof}


\end{document}